\documentclass[10pt]{article}
\usepackage{amssymb,color}
\usepackage{amsmath,graphicx,enumerate,textcomp}
\usepackage[abbrev]{amsrefs}
\usepackage[a4paper,top=1.in, bottom=1.5in, left=.8in, right=.8in, headsep=0.in, centering]{geometry}
\usepackage{amssymb,amsfonts,amsthm}
\usepackage{epsfig}
\numberwithin{equation}{section}
\usepackage{lipsum}
\usepackage{graphicx}
\usepackage{esint}

\newcommand{\eps}{\varepsilon}

\newtheorem{theorem}{Theorem}[section]
\newtheorem{lemma}{Lemma}[section]

\newtheorem{remark}{Remark}[section]
\newtheorem{proposition}{Proposition}[section]
\newtheorem{definition}{Definition}[section]

\title{On shifting the  thermal explosion threshold by a vortical flow in dimension two}

\author{ Tianyi Guo
\thanks{Department of Mathematical Sciences, 
Kent State University,
 Kent, OH 44242, USA. E-mail: {\tt tguo2@kent.edu  }}
\and Peter  V. Gordon
\thanks{Department of Mathematical Sciences, 
Kent State University,
 Kent, OH 44242, USA. E-mail: {\tt gordon@math.kent.edu}}
 }

\begin{document}
\maketitle

\begin{abstract}
This paper is concerned with a  study of  a natural generalization of  a classical Frank-Kamenetskii model of thermal explosion
 in the presence of a vortical flow in a two dimensional setting.
 This model describes  possible stationary temperature distributions  in a combustion vessel which boundary is maintained at a constant temperature.
 The model constitutes a Dirichlet boundary value 
 problem for a certain semi-linear elliptic equation that depends on a parameter $\lambda,$ called  Frank-Kamenetskii parameter.
  A remarkable property of this problem  is that it admits a classical minimal solution when
 the Frank-Kamenetskii parameter does not exceed some critical value  $\lambda^*$ and no classical solutions for $\lambda>\lambda^*$.
 The absence of a classical solution, in the framework  of  Frank-Kamenetskii  theory,  is associated with a thermal explosion event.
  Consequently, in  the context of combustion,  $\lambda^*,$ commonly called  an explosion threshold, is a maximal value of the Frank-Kamenetskii parameter which  allows to attain a thermal equilibrium within a combustion vessel and thus provides a sharp characterization of  the thermal explosion. A critical  temperature distribution corresponding to $\lambda^*$ is called an extremal solution.
 
 In this paper, we  show that, under an assumption of sufficiently fast growth of the reaction term,  
 there exists  a regular vortical flow that allows to adjust an explosion threshold by reversing its  direction,  provided a combustion vessel is not a disk.
 We also give rather detailed description of extremal solutions. In particular, we show that extremal solutions are always classical.

\end{abstract}

\medskip


\bigskip


\section{Introduction}\label{s:1}

In this paper we discuss some qualitative properties of solutions for  the following boundary value problem:
\begin{eqnarray}\label{eq:i1}
\left\{
\begin{array}{lll}
-\Delta u+A {\bf v} \cdot \nabla u=\lambda f(u) & \mbox{in} & \Omega, \\
u=0 & \mbox{on} & \partial \Omega,
\end{array}
\right.
\end{eqnarray}
where, $\Omega \in \mathbb{R}^2$ is an open bounded set with a sufficiently smooth (at least  $C^{2,\alpha}$)   boundary $\partial \Omega;$     $\lambda>0, ~A \in \mathbb{R}$ are parameters of the problem;
$f$ is the reaction term and ${\bf v}$ is a prescribed velocity  field satisfying the following hypothesis:
\begin{itemize}
\item (H1) The reaction term $f$ is a positive,  convex increasing  $C^2$ function of sufficiently fast growth at infinity. Namely, $f^{\prime}(s), f^{\prime\prime}(s)>0$ and $f^{\prime\prime}(s)$ is strictly increasing on $(0,\infty).$

\item (H2) The  $C^1$  vector field ${\bf v}=(v_x,v_y)$ prescribes an  incompressible vortical flow ${\bf v}=\nabla^{\perp} \psi=(\frac{\partial}{\partial y}, -\frac{\partial}{\partial x}  ) \psi,$ where $\psi$ is a  $C^2$ stream function and
satisfies no penetration condition  ${\bf v}\cdot {\bf \nu}=0,$ where $\nu$ is a unit normal to $\partial \Omega$.

\end{itemize}

Problem \eqref{eq:i1} falls into a general class of Gelfand type problems (see e.g.  \cite{Gelfand,Dupaigne}), which in turn are generalizations of a classical  model  of the stationary theory of a  thermal explosion derived in 1939 by D.A. Frank-Kamenetskii \cite{FK39}, see \cite{FK,ZBLM} for more details.
The latter is problem \eqref{eq:i1} with $A=0$ and $f(u)=e^u$. In the context of Frank-Kamenetskii theory,  $u$ and $f$ are an appropriately normalized temperature and the reaction rate
in a combustion vessel which boundary is maintained at a constant temperature (cold boundary). The parameter $\lambda>0$ measures the reaction strength and is called Frank-Kamenetskii
parameter.

Over the past 85 years, problem \eqref{eq:i1}  with $A=0$ received a considerable attention of engineers, physicists and mathematicians. Some of the  major mathematical results concerning
this problem can be found \cite{Gelfand,Fujita,CR75,KC67,KK74,Brezis1,Brezis2,Nedev,Martel,Cabre}.  The list is  far from being  complete.
 We refer to a book \cite{Dupaigne} for review 
of  mathematical results and \cite{FK,ZBLM,Sem,Law,Will} for the derivation, various applications and generalizations of this  problem in the context of combustion.

One of  the most distinctive feature of problem \eqref{eq:i1} with $A=0$ considered in a smooth bounded domain $\Omega\in\mathbb{R}^n$ is that there exists a critical value $0<\lambda^*<\infty$ such that this problem admits a  classical (generally not unique) solution
  for $\lambda\in (0,\lambda^*);$  the solution (possibly weak)  for $\lambda=\lambda^*$; and no solutions, even in a weak sense,  for $\lambda>\lambda^*$  \cite{Brezis1}.
\footnote{This result is very general and holds for \eqref{eq:i1} with $A=0$  considered in a smooth bounded domain $\Omega\in \mathbb{R}^n$ with a more general assumption
on the nonlinear term: $f$ is a positive non-decreasing convex $C^1$  function such that $\int_0^{\infty} ds/f(s) <\infty.$ }
From a physical perspective, this means that there is a critical value of Frank-Kamenetskii parameter $\lambda^*$  such that the chemical 
reaction within the combustion vessel can be balanced by the diffusion and a cold boundary,  provided $\lambda\in (0,\lambda^*]$, whereas when the reaction strength is above
 its critical value ($\lambda>\lambda^*$), no such balance is possible. One can show, see  \cite{Fujita, Brezis1}, that the  solution of parabolic version of problem \eqref{eq:i1} with zero initial
 condition will approach, as time goes to infinity, to the minimal (smallest) solution of \eqref{eq:i1} provided the latter exists and blows up in a finite time otherwise.
 This, in turn, implies that the temperature distribution within combustion vessel assumes a stationary profile for a Frank-Kamenetskii parameter at or below its
 critical value. This situation corresponds to a process of a slow oxidation. In contrast, when Frank-Kamenetskii parameter exceeds its critical value,  the intensive
 chemical reaction can't be suppressed by the diffusion and a cold boundary, which ultimately leads to a thermal explosion (thermal runaway).
 Consequently, the critical Frank-Kamenetskii parameter $\lambda^*$ provides a sharp characterization of a thermal explosion and hence  is typically called an explosion threshold.

 From the discussion above, it is clear that  minimal solutions of problem \eqref{eq:i1} with $A=0$ play a key role in the analysis of a spontaneous thermal explosion.
 From a mathematical perspective, minimal solutions are the most regular solutions that problem can produce. On the top of being classical
 for $\lambda\in (0,\lambda^*),$ they are the only semi-stable solutions  of \eqref{eq:i1} \cite{Brezis1,Brezis2}. Semi-stability means that the principal eigenvalue
 of  linearization  of \eqref{eq:i1} about these solutions is always non-negative. The pointwise limit of minimal solutions as $\lambda\nearrow \lambda^*$
 is called an extremal solution. The extremal solution is the  unique solution of \eqref{eq:i1} with $A=0$  \cite{Martel}, but may not be classical. The regularity of the extremal solution attracted considerable attention in the past decade. Milestone  results in this direction were obtained in \cite{Brezis2,Nedev}. This problem
 was recently fully  settled in \cite{Cabre}. Specifically,  it  was  shown in \cite{Cabre}  that  under a minimal  assumption on   $f,$ the  extremal solutions
 for \eqref{eq:i1} with $A=0$ considered in $\Omega \in \mathbb{R}^n$ 
 are  classical provided $n\le 9$.  This result  is sharp, as there is an explicit example of an extremal solution which does not belong to $L^{\infty}$ in the dimension
 $n=10$.

 Problem \eqref{eq:i1} in the presence of the advection term $(A \ne 0)$ was apparently introduced for the first time in \cite{Kagan} in order to understand the impact
 of stirring on thermal explosion limits.  A comprehensive  mathematical analysis of this problem in a rather general setting was performed in \cite{Kiselev}. 
 The results of \cite{Kiselev} in particular imply that  problem \eqref{eq:i1} under hypotheses (H1) and (H2) preserve all major qualitative features
 of the classical Gelfand problem. Indeed, it was shown that for a given $\Omega$,   and fixed $f, {\bf v},$   there exists $0<\lambda^*(A)<\infty$ such that problem \eqref{eq:i1}
 admits a minimal positive classical semi-stable solution $u_{\lambda}^{\sharp}$ provided $\lambda\in(0,\lambda^*(A))$ and no classical solution for $\lambda>\lambda^*(A)$.
 Moreover, $\lambda^*(A)$ is uniformly bounded from  below by a positive constant which depends only on  the domain  $\Omega$ and nonlinearity $f,$ and from above
 by a constant which depends on $\Omega$, $f$ and the stream function $\psi,$ but not on the amplitude of the flow. 
 In addition, \cite{Kiselev} contains a  detailed analysis  of  the limiting behavior of \eqref{eq:i1} with  $A\to \infty.$ In the following section,  we will recall several
  results of \cite{Kiselev} relevant to the present work. We note that in analogy to the classical Gelfand problem,
  one can define an extremal solution $u^*$  for problem \eqref{eq:i1}. Namely,
  \begin{eqnarray}\label{eq:i2}
  u^*(x) := \lim_{\lambda\nearrow \lambda^*} u_{\lambda}^{\sharp} (x).
  \end{eqnarray} 
 The regularity and qualitative properties  of the extremal solutions for \eqref{eq:i1} satisfying hypothesis (H1) and (H2)  will be discussed in details in the proceeding sections.

  Numerical studies of problem \eqref{eq:i1}  presented in \cite{Kagan} show that  the dependency of the critical Frank-Kamenetskii parameter $\lambda^*$ on $A$  is quite non-trivial.
  In particular, the dependency of $\lambda^*(A)$ is not monotone (see Fig.15 in \cite{Kagan}).
  Physically, this result is not very surprising. Indeed,  the flow can either  increase the  critical Frank-Kamenetskii parameter by enhancing the diffusion; or reduce the explosion threshold by creating a  hot spot.

This observation suggests the following natural question. Given a smooth  domain $\Omega\in \mathbb{R}^2$ and the reaction term $f$ satisfying hypothesis (H1),
is it possible to create a flow ${\bf v}$ of a small amplitude $A$   satisfying (H2)  such that it will decrease or increase the explosion threshold 
by  reversing
the direction of the flow?  This question is very natural from the perspective of applications as adjustment of the explosion threshold by a flow plays an
essential role in certain combustion devices. The principal goal of this paper is to  answer this question.
 Our main result gives a  positive answer to this question under the assumption that $\Omega$ is not a disk. Specifically, we  prove the following result.
 
 \begin{theorem} \label{t:1}Fix $f$ satisfying  (H1) and assume that $\Omega$ is not a disk.  Then, there exists a stream function $\psi$ generating an incompressible vortical  flow  ${\bf v}$ satisfying (H2) 
 such that
 \begin{eqnarray}\label{eq:i3}
 \lambda^*(A)=\lambda^*(0)+\theta A+o(A), 
 \end{eqnarray}
 with some $\theta=\theta(\psi)>0.$
\end{theorem}
 The proof of this result is nontrivial as incompressibility condition imposes substantial restrictions of possible choices of a flow. 
 We note that the proof is constructive
 and gives a direct recipe of how to design a  flow (or rather a stream function) such that  a  critical Frank-Kamenetskii parameter has a  property \eqref{eq:i3}. 
 In particular, if one assumes in addition to (H1) that $f\in C^3$, there is an explicit expression for  a stream function which generates a flow satisfying 
 the assumption (H2) that guarantees an optimal switching regime (see Remark \ref{r:1}).
 
 Let us note that in the case when $\Omega$ being a disk (or more generally a ball in $\mathbb{R}^n$), an incompressible flow can only increase the explosion threshold as follows from \cite[Theorem 1.2]{Novikov},
 hence switching by the flow is impossible in this case when changing the sign of $A$ even when $A$ is large. A similar effect is observed in the studies of exit times of diffusions with incompressible drift, see \cite{Iyer}  and references therein.

 The proof of Theorem \ref{t:1}  is based, in part,  on several results which provide rather detailed  information on the  behavior 
 of extremal solutions of problem \eqref{eq:i1} satisfying (H1) and (H2). 
 These results are  of independent interest and summarized as follows.
 
 \begin{proposition}\label{p:1}  Fix the domain  $\Omega$, the reaction function $f$  satisfying  (H1) and the flow ${\bf v}$ satisfying  (H2). 
 Then, 
 \begin{itemize}
 \item[1)] For any  $A \in \mathbb{R},$  the  extremal solution $u^*(\cdot, A)$  of \eqref{eq:i1} is classical; 
 \item [2)] For any  $A \in \mathbb{R},$ the extremal solution  $u^*(\cdot, A)$  is the unique classical solution of \eqref{eq:i1} for $\lambda=\lambda^*(A);$
 \item [3)]  The linearization of \eqref{eq:i1} on the extremal solution $u^*$  is degenerate.  That is, the principal eigenvalue of  ${\cal L}_*:=-\Delta +A {\bf v} \cdot \nabla -\lambda^*(A)  f^{\prime} (u^*)$
 with Dirichlet boundary conditions is zero;
  \item[4)]  The critical  Frank-Kamenetskii  parameter $\lambda^*(A)$ is continuous and differentiable with respect to $A$.
\end{itemize} 
 \end{proposition}
 

The paper is organized as follows. In the next section, we recall  several known results  adopted to the setting considered in this paper and
prove relatively standard  lemmas needed for this work. In  Section \ref{s:3},   we give a proof of Proposition \ref{p:1}.    
Section \ref{s:4} contains a proof of the main result of this paper.

\section{Preliminaries}\label{s:2}

In this section, we state several known results and definitions  which will be used in the following sections. We start with recalling results from \cite{Kiselev} relevant to the present work.

Fix the domain $\Omega,$ the reaction function $f,$ and the stream function $\psi$ satisfying (H1) and (H2) respectively. Then, the solutions of   problem \eqref{eq:i1} have the following properties.

\medskip

\noindent {\bf i)} There exists an extremal value of Frank-Kamenetskii parameter $0<\lambda^*(A)<\infty$  such that problem  \eqref{eq:i1} admits  a minimal classical  solution $u^{\sharp}_{\lambda}$   for all $\lambda\in (0,\lambda^*(A))$
and no classical solutions for $\lambda>\lambda^*(A).$
\medskip

\noindent {\bf ii)}   Minimal solutions $u^{\sharp}_{\lambda}$ form an increasing family of functions parametrized by $\lambda$ for  $\lambda\in (0,\lambda^*(A)).$
\medskip

\noindent {\bf iii) }  The principal eigenvalue of linearized  problem \eqref{eq:i1} on the minimal solution for $\lambda\in(0,\lambda^*(A))$  is positive. That is, an eigenvalue problem:
\begin{eqnarray}\label{eq:p1}
\left\{
\begin{array} {lll}
{\cal L}(\phi)=\mu \phi &\mbox{in} & \Omega,\\
\phi=0 & \mbox{on} & \partial \Omega,
\end{array}
\right.
\end{eqnarray} 
with 
\begin{eqnarray} \label{eq:p2}
{\cal L}:= -\Delta+A {\bf v} \cdot \nabla -\lambda f^{\prime} (u^{\sharp}_{\lambda})
\end{eqnarray}
has its principal eigenvalue $\mu_1(\lambda)> 0,$ provided  $\lambda\in(0,\lambda^*(A)).$
\medskip

\noindent {\bf iv)} There exist a constant $\lambda_0>0$ which depends on $\Omega$ and $f,$ and a constant $\lambda_{\infty}<\infty$ which depends on $\Omega$, $f$ and the stream function $\psi$, but not on $A$ such that
$\lambda_0<\lambda^*(A)<\lambda_{\infty}$ for all $A\in\mathbb{R}.$
\medskip

Properties  {\bf i)}, {\bf ii)} and {\bf iii)} follow from \cite[Proposition 1.1]{Kiselev}. The lower bound on $\lambda^*(A)$  in {\bf iv)} was established in \cite[Theorem 1.2]{Kiselev}.
The upper bound on $\lambda^*(A)$  in {\bf iv)} follows from \cite[Theorem 1.4]{Kiselev} combined with an observation  that any  stream function $\psi$ satisfying (H2)  is the first integral of a  flow  ${\bf v} \in H_0^1(\Omega)$.
We recall that the first integral of a flow ${\bf v} \in H_0^1(\Omega)$ is a scalar function $w\in H^1(\Omega)$ such that ${\bf v}\cdot\nabla  w=0$ almost everywhere in $\Omega$
(see \cite[Defintion 0.1]{eigenvalue}).  

The construction of  the  minimal  solutions  \eqref{eq:i1} heavily relies on the method of super solutions. A super-solution for \eqref{eq:i1} is  defined as follows.
\begin{definition}
We say that $\bar u\in C^2(\Omega) \cap C(\bar \Omega),$  $\bar u>0$ in $ \Omega$ is a classical super-solution  of \eqref{eq:i1} provided
\begin{eqnarray}\label{eq:p3}
\left\{
\begin{array}{lll}
-\Delta \bar u +A {\bf v} \cdot \nabla \bar u \ge \lambda f(\bar u) & \mbox{in} & \Omega,\\
\bar u =0 & \mbox{on} & \partial \Omega.
\end{array}
\right.
\end{eqnarray}
\end{definition}
The presence of  a super-solution guarantees the existence of the minimal solution for \eqref{eq:i1}, namely  the following result holds.
\begin{lemma} \label{l:p1}
Assume that \eqref{eq:i1} admits a positive classical super-solution, then \eqref{eq:i1} admits a minimal classical solution $u^{\sharp}_{\lambda}\in C^{2,\alpha}(\bar \Omega).$ 
\end{lemma}
A proof of this lemma is based on Sattinger's  monotone iterations arguments \cite[Theorem 2.1]{Satt} and is applicable to a wide class of elliptic boundary value problems which includes \eqref{eq:i1}. A brief sketch of these  arguments for \eqref{eq:i1} is  given in \cite[Lemma 2.3]{Kiselev}.


Let us also note that   the principal  eigenvalue $\mu_1(\lambda)$  of \eqref{eq:p2} is real and  simple. Hence,  the corresponding
eigenfunction  $\phi_1>0$ in $\Omega$, as follows from \cite[Theorem 2.1, Corollaries 2.1-2.3] {BNV}. Moreover, by \cite[Theorem 6.15]{GT}, 
$\phi_1\in C^{2,\alpha} (\bar \Omega).$

  The linear stability condition for minimal solutions of \eqref{eq:i1} (property ${\bf iii)}$ above)  in the  case when $A=0$,   implies that
 $\int_{\Omega} |\nabla \eta|^2 \ge \lambda \int_{\Omega}  f^{\prime} (u^{\sharp}_\lambda) \eta^2$ for all $\eta \in H_0^1(\Omega)$ and $\lambda\in(0,\lambda^*(0)),$ thanks to the variational characterization of the principal eigenvalue in this case \cite[Section 8.12]{GT}. This condition is often called  semi-stability condition.
 It is well known that there is no  variational characterization of the principal eigenvalue for \eqref{eq:p1}  in the case when $A\ne0$ and  ${\bf v}$ is a vortical flow. 
However, it is still possible to construct an integral relation involving appropriate class of test functions which, in some sense,  mimics the semi-stability condition.
Namely, the following  result holds.

\begin{lemma} \label{l:p2}
Fix the domain $\Omega$,  the reaction $f$ satisfying (H1)  and a vortical flow ${\bf v}$ satisfying (H2) which is  normalized such that  
$||{\bf v}||_{\infty}=1.$   Let $\lambda\in (0,\lambda^*(A))$  and  $u^{\sharp}_{\lambda}$ be the minimal solution of \eqref{eq:i1}.
Then, for any $\alpha\in(0,1)$ and any $\eta \in H_0^{1}(\Omega),$ we have,
\begin{eqnarray} \label{eq:p4}
\int_{\Omega}|\nabla \eta|^2 \ge \alpha \lambda \int_{\Omega} f^{\prime} (u^{\sharp}_{\lambda}) \eta^2 -\frac{\alpha A^2}{4(1-\alpha)}  \int_{\Omega} \eta^2.
\end{eqnarray}
\end{lemma}  
 Inequality \eqref{eq:p4} will play a pivotal role in the proof of Proposition \ref{p:1}. 
 A proof of this inequality is given in \cite[Lemma 2 and Theorem 2]{semistab}, but we provide it
 here for completeness. We also note that \eqref{eq:p4} was used in proving the regularity of the extremal solutions for problem \eqref{eq:i1} with singular nonlinearities \cite{LYZ}.

\begin{proof}[Proof of Lemma \ref{l:p2}]
Fix $\eps>0,$ 
 take an arbitrary $\eta \in H_0^1(\Omega)$ 
 and  set $\xi:= \frac{\eta }{(\eps+\phi_1)^{\alpha}} \in H_0^1(\Omega),$  where $\phi_1>0$ in $\Omega$ and normalized such that $||\phi_1||_{\infty}=1$ is the  eigenfunction corresponding to the  principal eigenvalue
of \eqref{eq:p2} in $\Omega.$ 
Observe that,
\begin{eqnarray}\label{eq:p5}
&&\int_{\Omega} |\nabla \eta |^2= \int_{\Omega} |\nabla \left( (\eps+\phi_a)^\alpha \xi) \right)|^2=\alpha^2 \int_{\Omega} (\eps+\phi_1)^{2(\alpha-1)} \xi^2 |\nabla \phi_1|^2+\nonumber\\
&& \alpha \int_{\Omega} (\eps+\phi_1)^{2\alpha-1} \nabla \phi_1 \cdot \nabla (\xi^2)+\int_{\Omega} (\eps+\phi_1)^{2\alpha} |\nabla \xi|^2\ge\\
&& \alpha^2 \int_{\Omega} (\eps+\phi_1)^{2(\alpha-1)} \xi^2 |\nabla \phi_1|^2+ \alpha \int_{\Omega} (\eps+\phi_1)^{2\alpha-1} \nabla \phi_1 \cdot \nabla (\xi^2). \nonumber
\end{eqnarray}
Inserting the definition of $\xi$ into  the first term in the right hand side of the expression above, we have
\begin{eqnarray}\label{eq:p6}
\int_{\Omega} |\nabla \eta |^2\ge \alpha^2 \int_{\Omega}\left(\frac{ |\nabla \phi_1| \eta}{\eps+\phi_1}\right)^2+\alpha \int_{\Omega} (\eps+\phi_1)^{2\alpha-1} \nabla \phi_1 \cdot \nabla (\xi^2).
\end{eqnarray}
Next, integrating by parts the second term on the right hand side of \eqref{eq:p6}  and taking into account that $\xi=0$ on $\partial \Omega,$ we obtain
\begin{eqnarray}\label{eq:p7}
&&\int_{\Omega} (\eps+\phi_1)^{2\alpha-1} \nabla \phi_1 \cdot \nabla (\xi^2)=\int_{\Omega} \nabla \cdot \left(  (\eps+\phi_1)^{2\alpha-1} \xi^2 \nabla \phi_1 \right)
-\int_{\Omega} \left(\nabla \cdot \left((\eps+\phi_1)^{2\alpha-1}  \nabla \phi_1 \right)\right)\xi^2=\nonumber \\
&& \int_{\partial \Omega}   (\eps+\phi_1)^{2\alpha-1} \xi^2 \nabla \phi_1 \cdot \nu-\int_{\Omega} \left(\nabla \cdot \left((\eps+\phi_1)^{2\alpha-1}  \nabla \phi_1 \right)\right)\xi^2=- \int_{\Omega} \left(\nabla \cdot \left((\eps+\phi_1)^{2\alpha-1}  \nabla \phi_1 \right)\right)\xi^2\\
&&= \int_{\Omega} (-\Delta \phi_1) (\eps+\phi_1)^{2\alpha-1} \xi^2
-(2\alpha-1) \int_{\Omega} (\eps+\phi_1)^{2(\alpha-1)} \xi^2 |\nabla \phi_1|^2. \nonumber 
 \end{eqnarray}
Inserting the definition of $\xi$ into the right hand side of the expression above, we have
\begin{eqnarray}\label{eq:p8}
\int_{\Omega} (\eps+\phi_1)^{2\alpha-1} \nabla \phi_1 \cdot \nabla (\xi^2)=\int_{\Omega} \frac{(-\Delta \phi_1)\eta^2}{\eps+\phi_1}
 -(2\alpha-1) \int_{\Omega} \left(\frac{|\nabla \phi_1|\eta}{\eps+\phi_1}\right)^2.
\end{eqnarray}
Combining  \eqref{eq:p6} and \eqref{eq:p8},
we have
\begin{eqnarray} \label{eq:p9}
\int_{\Omega} |\nabla \eta |^2\ge \alpha \int_{\Omega} \frac{(-\Delta \phi_1)\eta^2}{\eps+\phi_1}+\alpha(1-\alpha)  \int_{\Omega} \left(\frac{|\nabla \phi_1|\eta}{\eps+\phi_1}\right)^2.
\end{eqnarray}\label {eq:p10}
By  property {\bf iii)} above, that is $ -\Delta \phi_1+A{\bf v} \cdot \nabla \phi_1 -\lambda f^{\prime} (u^{\sharp}_{\lambda} ) \phi_1>0$ in $\Omega.$  Hence, \eqref{eq:p9} yields
\begin{eqnarray}\label{eq:p11}
\int_{\Omega} |\nabla \eta |^2 \ge \lambda \alpha \int_{\Omega} f^{\prime} (u^{\sharp}_{\lambda}) \left(\frac{ \phi_1}{\eps+\phi_1} \right) \eta^2-\alpha A\int_{\Omega}
\left(\frac{ {\bf v}\cdot \nabla \phi_1} {\eps+\phi_1} \right) \eta^2+\alpha(1-\alpha)  \int_{\Omega} \left(\frac{|\nabla \phi_1|\eta}{\eps+\phi_1}\right)^2.
\end{eqnarray}
 Using Cauchy-Schwartz inequality, the normalization $||{\bf v}||_{\infty}=1$ and   the expression above, we have
\begin{eqnarray}\label{eq:p12}
\int_{\Omega} |\nabla \eta |^2 \ge \lambda \alpha \int_{\Omega} f^{\prime} (u^{\sharp}_{\lambda}) \left(\frac{ \phi_1}{\eps+\phi_1} \right) \eta^2-\alpha |A|\int_{\Omega} \left(\frac{ |\nabla \phi_1|}{\eps+\phi_1} \right) \eta^2+\alpha(1-\alpha)  \int_{\Omega} \left(\frac{|\nabla \phi_1|\eta}{\eps+\phi_1}\right)^2.
\end{eqnarray}
Since,
\begin{eqnarray}\label{eq:p13}
-|A| \frac{|\nabla \phi_1|}{\eps+\phi_1} +(1-\alpha) \left( \frac{|\nabla \phi_1|}{\eps+\phi_1} \right)^2= 
\left( \sqrt{1-\alpha} \frac{|\nabla \phi_1|}{\eps+\phi_1} -\frac{|A|}{2\sqrt{1-\alpha}}\right)^2-\frac{A^2}{4(1-\alpha)},
\end{eqnarray}
we obtain from \eqref{eq:p12}
\begin{eqnarray}\label{eq:p14}
\int_{\Omega} |\nabla \eta |^2 \ge \lambda \alpha \int_{\Omega} f^{\prime} (u^{\sharp}_{\lambda}) \left(\frac{ \phi_1}{\eps+\phi_1} \right) \eta^2
-\frac{\alpha A^2}{4(1-\alpha)}
\int_{\Omega} \eta^2.
\end{eqnarray}
Since  $\eps$ is arbitrarily small,  we obtain \eqref{eq:p4} from \eqref{eq:p14} taking $\eps\to 0$.

\end{proof}

Another important inequality  needed for  the proof of Proposition \ref{p:1} is as follows.

\begin{lemma} \label{l:p3} Fix the reaction term $f$ and assume that (H1) holds. Then, there exists $\delta>0$ such that
\begin{eqnarray}
\int_{t_1}^{t_2} \left (f^{\prime}(s) \right)^2 ds \le (1-\delta) f^{\prime} (t_2) ( f(t_2)-f(t_1)) \quad \forall t_2>t_1\ge 0.
\end{eqnarray}
This result follows from  the  proofs  of \cite[Theorem 1.3]{GMN} and   \cite[Lemma 5.2]{GMN}.

\end{lemma}

Next,  we will need the following two results to prove the uniqueness of a classical solution for \eqref{eq:i1} with $\lambda=\lambda^*$. These results
are direct adaptations of Lemmas 2.1 and 2.2  from \cite{Martel}.

\begin{lemma} \label{l:p4}
Assume that (H1) and (H2) hold and a problem
\begin{eqnarray}
\left\{
\begin{array}{lll}
-\Delta w+A {\bf v} \cdot \nabla w=\tilde \lambda f(w)+\eps & \mbox{in} & \Omega, \\
w=0 & \mbox{on} & \partial \Omega,
\end{array}
\right.
\end{eqnarray}
admits a classical solution for some $\tilde \lambda>0$ and $\eps>0.$ Then, problem \eqref{eq:i1} admits a classical solution for $\lambda=\tilde \lambda (1+\gamma)$ for some
$\gamma>0.$

\end{lemma}

\begin{proof} 
Let $\chi$ be the solution of
\begin{equation}\label{chi}
\left\{  
\begin{aligned}
&-\Delta \chi +A{\bf v}\cdot \nabla \chi =1 &&\text{in }\Omega, \\
&\chi=0 &&\text{on } \partial\Omega.
\end{aligned}
\right.
\end{equation}
By Hopf's  Lemma, see e.g  \cite[Section 6.4.2 ]{Evans}, there exists $\beta_{0}>0$ such that $w-\beta\chi>0$  in $\Omega,$ 
 $\forall \beta\in (0,\beta_{0}).$ 
 Set
\begin{equation}
			z:=(1+\gamma)w-\beta\chi.
\end{equation}
Then,
\begin{equation}\label{eq:T1}
\left\{  
\begin{aligned}
&-\Delta z+A{\bf v}\cdot \nabla z=\tilde{\lambda}(1+\gamma)f(w)+(1+\gamma)\eps-\beta &&\text{in } \Omega, \\
&z=0 &&\text{on } \partial\Omega.
\end{aligned}
\right.
\end{equation}
Choosing $\beta= \min\{\frac{\beta_{0}}{2},\eps\}$, we have from \eqref{eq:T1}
\begin{equation}
\left\{  
\begin{aligned}
&-\Delta z+A{\bf v}\cdot \nabla z > \tilde{\lambda}(1+\gamma)f(w) &&\text{in } \Omega, \\
&z=0 &&\text{on } \partial\Omega.
\end{aligned}
\right.
\end{equation}
Using  Hopf's Lemma again, we chose  $\gamma>0$ such that $\gamma w-{\beta} \chi<0$ in $\Omega$.
With this choice of the parameters $\beta, \gamma >0$ we have  that $0<z<w$ in $\Omega$. 
Since $f$ is an increasing function with $f(0)>0$, we have
\begin{equation}
\left\{  
\begin{aligned}
&-\Delta z+A{\bf v}\cdot \nabla z >  \tilde{\lambda}(1+\gamma)f(w)>\tilde{\lambda}(1+\gamma)f(z) &&\text{in } \Omega, \\
&z=0 &&\text{on } \partial\Omega.
\end{aligned}
\right.
\end{equation}
Therefore, $z$ is a bounded classical super solution for the following problem  
\begin{equation}
\left\{  
\begin{aligned}
&-\Delta u+A{\bf v}\cdot \nabla u = \tilde{\lambda}(1+\gamma)f(u) &&\text{in } \Omega, \\
&u=0 &&\text{on } \partial\Omega,
\end{aligned}
\right.
\end{equation}
which implies the existence of a classical solution for \eqref{eq:i1} with $\lambda=\tilde \lambda(1+\gamma),$ thanks to Lemma \ref{l:p1}.
	\end{proof}

\begin{lemma}\label{l:p5}
 Assume that $\bar w$ is a classical super-solution of \eqref{eq:i1} with $\lambda=\lambda^*.$ That is,
\begin{eqnarray}\label{eq:T2}
\left\{
\begin{array}{lll}
-\Delta\bar  w+A {\bf v} \cdot \nabla \bar w \ge  \lambda^*  f(\bar w) & \mbox{in} & \Omega, \\
\bar w=0 & \mbox{on} & \partial \Omega,
\end{array}
\right.
\end{eqnarray}
Then, $\bar w$ solves,
\begin{eqnarray}
\left\{
\begin{array}{lll}
-\Delta\bar  w+A {\bf v} \cdot \nabla \bar w =  \lambda^*  f(\bar w) & \mbox{in} & \Omega, \\
\bar w=0 & \mbox{on} & \partial \Omega.
\end{array}
\right.
\end{eqnarray}
\end{lemma}

\begin{proof} Assume   \eqref{eq:T2} holds. Then, there exists  $g\in C^2(\Omega) \cap C(\bar \Omega),$ $g \geq 0$ and $g \not\equiv 0$  in $\Omega$ such that
\begin{equation}
-\Delta \bar{w}+A{\bf v}\cdot \nabla \bar{w} = \lambda^{\ast}f(\bar{w})+g.
\end{equation}
Let $\zeta$ be the solution of
\begin{equation}
\left\{  
\begin{aligned}
&-\Delta \zeta + A{\bf v}\cdot \nabla \zeta=g &&\text{in } \Omega, \\
&\zeta=0 &&\text{on } \partial\Omega,
\end{aligned}
\right.
\end{equation}
and  $\chi$ be  the  solution of  (\ref{chi}).  Arguing as in the previous lemma, we can choose two constants $0<\beta<1$ and $\eps >0$ such that
$\bar w-\beta\zeta>0$ and $\eps \chi-\beta  \zeta <0$ in $\Omega.$
Set
\begin{equation}
z:=\bar w+\eps \chi-\beta\zeta,
\end{equation}
and observe that
\begin{equation}\label{eq:T3}
\left\{  
\begin{aligned}
&-\Delta z+A{\bf v}\cdot \nabla z=\lambda^{\ast}f(\bar{w})+g+\eps-\beta g \ge \lambda^{\ast}f(\bar{w})+\eps&&\text{in } \Omega, \\
&z=0 &&\text{on } \partial\Omega,
\end{aligned}
\right.
\end{equation}
where the last inequality holds as $g\ge 0$ and $\beta<1$. By our choice of the parameters $\beta$ and $\eps$ we have that $0<z<\bar w$ in $\Omega.$  Hence, by \eqref{eq:T3} 
we have
\begin{equation}
\left\{  
\begin{aligned}
&-\Delta z+A{\bf v}\cdot \nabla z \ge \lambda^{\ast}f(z)+\eps &&\text{in } \Omega, \\
&z=0 &&\text{on } \partial\Omega.
\end{aligned}
\right.
\end{equation}
Therefore, by Lemma \ref{l:p1} there exists $\tilde z$ a classical solution of
 \begin{equation}
\left\{  
\begin{aligned}
&-\Delta \tilde z+A{\bf v}\cdot \nabla \tilde z=  \lambda^{\ast}f(\tilde z)+\eps &&\text{in } \Omega, \\
&z=0 &&\text{on } \partial\Omega.
\end{aligned}
\right.
\end{equation}
By Lemma \ref{l:p4} we thus have that there is $\gamma>0$ such that 
		
\begin{equation}
\left\{  
\begin{aligned}
&-\Delta u+A{\bf v}\cdot \nabla u = \lambda^{\ast}(1+\gamma)f(u) &&\text{in } \Omega, \\
&u=0 &&\text{on } \partial\Omega,
\end{aligned}
\right.
\end{equation}
admits a classical solution which contradicts the definition of $\lambda^*.$
\end{proof}

\section{Properties of extremal solutions: proof of Proposition \ref{p:1}}\label{s:3}

In this section, we present a proof of Proposition \ref{p:1} and describe some additional properties of the extremal solutions for \eqref{eq:i1}. 
The following result shows that the extremal solutions are classical.
\begin{lemma} \label{l:31}
 Fix the domain $\Omega \in \mathbb{R}^2$, the reaction function $f$ satisfying (H1) and normalized such that $f(0)=1$, the flow ${\bf v}$ satisfying (H2) and normalized such that $||{\bf v}||_{\infty}=1,$
and an amplitude of the flow $A.$ Then, the extremal solution $u^*$  of \eqref{eq:i1} defined  in \eqref{eq:i2} is $C^{2,\alpha} (\bar \Omega).$
\end{lemma}

\begin{proof}
The proof  follows a general strategy of \cite[Theorem 1]{Nedev}. 
Let $\lambda\in (0,\lambda^*),$ set $u:=u_{\lambda}^{\sharp}$  and
\begin{eqnarray}\label{eq:m1}
\tilde f(u)=f(u)-1, \quad g(u)=\int_0^u \left ( f^{\prime} (s) \right)^2 ds.
\end{eqnarray}
Multiplying \eqref{eq:i1} by $g(u)$ and integrating the result by parts and taking into account incompressibility and no-penetration conditions on ${\bf v}$ (see (H2)) 
we have
\begin{eqnarray}\label{eq:m2a}
\int_{\Omega} \left(f^{\prime} (u) |\nabla u|\right)^2=\lambda \int_{\Omega} \tilde f(u) g(u) +\lambda \int_{\Omega} g(u).
\end{eqnarray}
By inequality  \eqref{eq:p4}  in Lemma \ref{l:31} with  $\eta=\tilde f(u),$   we have that for any $\beta \in (0,1)$  the following holds 
\begin{eqnarray}\label{eq:m3a}
\int_{\Omega} \left(f^{\prime} (u) |\nabla u|\right)^2\ge \beta \lambda \int_{\Omega} f^{\prime}(u) \left( \tilde f(u) \right)^2-\frac{\beta A^2}{4(1-\beta)} \int_{\Omega} \left( \tilde f(u) \right)^2.
\end{eqnarray}
Combining \eqref{eq:m2a} and \eqref{eq:m3a} yields
\begin{eqnarray}\label{eq:m4}
\lambda \int_{\Omega} \tilde f(u) g(u) +\lambda \int_{\Omega} g(u)\ge \beta \lambda \int_{\Omega} f^{\prime}(u) \left( \tilde f(u) \right)^2-\frac{\beta A^2}{4(1-\beta)} \int_{\Omega} \left( \tilde f(u) \right)^2.
\end{eqnarray}
Rearranging terms in the expression above, we obtain
\begin{eqnarray}\label{eq:m5}
\beta  \int_{\Omega} f^{\prime}(u) \left( \tilde f(u) \right)^2\le \int_{\Omega} \tilde f(u) g(u) +\int_{\Omega} g(u)+\frac{\beta A^2}{4\lambda(1-\beta)} \int_{\Omega}\left( \tilde f(u) \right)^2.
\end{eqnarray} 
Next, observe that by Lemma \ref{l:p3} 
\begin{eqnarray}\label{eq:m6}
g(u)\le (1-\delta) f^{\prime}(u)\tilde f(u),
\end{eqnarray}
for some $\delta>0$.
Applying this estimate  to the first term on the right hand side of   \eqref{eq:m5},  we have
\begin{eqnarray}
\beta  \int_{\Omega} f^{\prime}(u) \left( \tilde f(u) \right)^2\le (1-\delta)  \int_{\Omega}  f^{\prime}(u) \left( \tilde f(u) \right)^2 +\int_{\Omega} g(u)+\frac{\beta A^2}{4\lambda(1-\beta)} \int_{\Omega}\left( \tilde f(u) \right)^2.
\end{eqnarray}
Rearranging therms in the inequality above, we obtain
\begin{eqnarray}\label{eq:m7}
(\beta+\delta-1)\int_{\Omega} f^{\prime}(u) \left( \tilde f(u) \right)^2\le  \int_{\Omega} g(u)+\frac{\beta  A^2}{4\lambda(1-\beta)} \int_{\Omega}\left( \tilde f(u) \right)^2.
\end{eqnarray}
Choosing $\beta=1-\frac{\delta}{2}$ in \eqref{eq:m7}, we have
\begin{eqnarray}\label{eq:m8}
\int_{\Omega} f^{\prime}(u) \left( \tilde f(u) \right)^2\le \frac{2}{\delta}  \int_{\Omega} g(u) +\frac{A^2}{\lambda \delta^2}  \int_{\Omega}\left( \tilde f(u) \right)^2.
\end{eqnarray}
We also note that by property {\bf iv)} in Section \ref{s:2},  $\lambda\ge \frac{\lambda_0}{2}>0$ for $\lambda\in \left [\frac{\lambda^*}{2},\lambda^*\right)$. Hence, in this range
of the parameter $\lambda,$  \eqref{eq:m8} reduces to
\begin{eqnarray}\label{eq:m9}
\int_{\Omega} f^{\prime}(u) \left( \tilde f(u) \right)^2\le \frac{2}{\delta}  \int_{\Omega} g(u) +\frac{2A^2}{\lambda_0 \delta^2}  \int_{\Omega}\left( \tilde f(u) \right)^2.
\end{eqnarray}
Observe, using \eqref{eq:m6},  that for $s>0$ 
\begin{eqnarray}\label{eq:m10}
 \frac{g(s)}{f^{\prime}(s) (\tilde f(s))^2} \le (1-\delta) \frac{f^{\prime}(s) \tilde f(s) }{f^{\prime}(s) (\tilde f(s))^2}=\frac{(1-\delta)}{\tilde f(s)}, \quad   \frac{(\tilde f(s))^2}{f^{\prime}(s) (\tilde f(s))^2} =\frac{1}{ f^{\prime} (s)}.
\end{eqnarray}
Hence, by (H1),  we have
\begin{eqnarray}\label{eq:m11}
 \frac{g(s)}{f^{\prime}(s) (\tilde f(s))^2},   \frac{(\tilde f(s))^2}{f^{\prime}(s) (\tilde f(s))^2} \to 0, \quad \mbox{as} \quad s\to \infty.
\end{eqnarray}
Consequently, passing to the limit $\lambda\nearrow \lambda^*$ in \eqref{eq:m9},  we observe that the left hand side in this expression dominates the right hand side for sufficiently
large $u$.
Therefore, both terms on  the right hand side \eqref{eq:m9}  remain  bounded as $\lambda\nearrow \lambda^*$. As a result, we have that the left hand side in  \eqref{eq:m9} 
is bounded in this limit, that is 
\begin{eqnarray} \label{eq:m12}
\int_{\Omega} f^{\prime}(u^*) \left( \tilde f(u^*) \right)^2\le C,
\end{eqnarray}
 for some constant $C<\infty$ that depends on  $\Omega, f$ and $A$. This observation immediately  implies that $f(u^*) \in L^2(\Omega).$ Therefore, by \cite[Theorem 8.12]{GT},
 $u^*\in W^{2,2} (\Omega),$ which in turn implies that $u\in C^{0,\alpha} (\bar \Omega),$  as follows from \cite[Theorem 4.12]{AF}.  Consequently,  under hypothesis (H1), $f(u^*)\in C^{0,\alpha}(\bar \Omega)$ and hence, by \cite[Theorem 6.14]{GT},  $u\in C^{2,\alpha}(\bar \Omega)$. As a result,  the extremal solution of \eqref{eq:i1} is classical as claimed. 
\end{proof}

The next result shows that the extremal solution of \eqref{eq:i1} is the unique classical solution for this problem at $\lambda=\lambda^*.$
\begin{lemma} \label{l:32}
Under the assumptions of Lemma \ref{l:31} the extremal solution $u^*$ of \eqref{eq:i1} is the only classical solution of  \eqref{eq:i1} with $\lambda=\lambda^*$.
\end{lemma}

\begin{proof}Let $\lambda=\lambda^*$ and assume that  in addition to  $u^*,$ the extremal solution of \eqref{eq:i1},  there exists
  another classical solution $w$ of this problem. By the construction of the extremal solution,  $w>u^*$ in $\Omega$. Set
\begin{equation}
z=\frac{u^{\ast}+w}{2},
\end{equation}
clearly $u^{\ast}<z<w$ in $\Omega.$ By (H1),  $f$  is strictly convex, hence $\frac{f(u^{\ast})+f(w)}{2}>f(\frac{u^{\ast}+w}{2})$. Therefore,
\begin{equation}
\left\{  
\begin{aligned}
&-\Delta z+A{\bf v}\cdot \nabla z = \frac{1}{2}\lambda^* (f(u^{\ast})+f(w)) >\lambda^* f(z) &&\text{in } \Omega, \\
&z=0 &&\text{on } \partial\Omega,
\end{aligned}
\right.
\end{equation}
Consequently, $z$ is a strict super-solution of \eqref{eq:i1} at $\lambda=\lambda^*$ which is in contradiction by Lemma \ref{l:p5}. 
\end{proof}

\begin{lemma}\label{l:33}Under  the assumptions of Lemma \ref{l:31}
an eigenvalue problem 
\begin{eqnarray}\label{eq:T6}
\left\{
\begin{array} {lll}
{\cal L}_*(\phi)=\mu \phi &\mbox{in} & \Omega,\\
\phi=0 & \mbox{on} & \partial \Omega,
\end{array}
\right.
\end{eqnarray} 
with
 ${\cal L}_*:=-\Delta +A {\bf v} \cdot \nabla -\lambda^*(A)  f^{\prime} (u^*)$
has its   the principal eigenvalue $\mu_1=0.$
\end{lemma}

\begin{proof} Assume $\mu_1>0,$   set $\lambda^*=\lambda^*(A)$ and
\begin{eqnarray}\label{eq:T4}
w=u^{\ast}+\eps(\phi_1+b\chi),
\end{eqnarray}
 where 
$\phi_1>0$  in $\Omega$ is the eigenfunction of ${\cal L}_*$  with Dirichlet boundary conditions corresponding to the principal eigenvalue $\mu_1$ normalized such that $||\phi_1||_{\infty}=1$, 
$\chi$ is the solution of  (\ref{chi}), $0<\eps\ll 1$ and $b>0$ is a parameter to be chosen later.
Define 
\begin{eqnarray}\label{eq:T5}
P(w):=-\Delta w +A{\bf v}\cdot \nabla w-\lambda^{\ast}f(w).
\end{eqnarray}
Substituting \eqref{eq:T4} into \eqref{eq:T5}, using  equations \eqref{eq:i1}, \eqref{eq:T6}, \eqref{chi} and taking into account  the boundary conditions
 for $u^*$, $\phi_1$ and $\chi,$ we have
\begin{equation}
\left\{  
\begin{aligned}
&P(w)=\eps \mu_1 \phi_1 +\eps b +\lambda^*(f(u^*)+\eps f^{\prime}(u^*) \phi_1-f(w))
 &&\text{in } \Omega, \\
&w=0 &&\text{on } \partial\Omega.
\end{aligned}
\right.
\end{equation}
Using the Taylor expansion of $f(w)$ about  $u^*$, we have 
$f(w)=f(u^{\ast}+\eps(\phi_1+b\chi))=f(u^{\ast})+\eps f^{\prime}(u^{\ast})(\phi_1+b\chi)+\mathcal{O}(\eps^2)$ and
therefore
$f(u^*)+\eps f^{\prime}(u^*) \phi_1-f(w)=-b\eps f^{\prime}(u^*)\chi +\mathcal{O}(\eps^2).$
Hence,
\begin{equation}
\left\{  
\begin{aligned}
&P(w)= \eps \left( \mu_1 \phi_1+ b -  b \lambda^{\ast} f'(u^{\ast})\chi\right)+ \mathcal{O}(\eps^2) &&\text{in } \Omega, \\
&w=0 &&\text{on } \partial\Omega,
\end{aligned}
\right.
\end{equation}
which can be  rewritten in a more convenient form as 
\begin{equation}\label{eq:T7}
\left\{  
\begin{aligned}
&P(w)= \eps \left( b +\phi_1\mu_1\left (1- 
 b\left(\frac{ \lambda^{\ast}}{\mu_1}\right)\left( \frac{\chi}{\phi_1}\right)f^{\prime}(u^{\ast})\right)\right)+ \mathcal{O}(\eps^2) &&\text{in } \Omega, \\
&w=0 &&\text{on } \partial\Omega.
\end{aligned}
\right.
\end{equation}
Next, note that Hopf's Lemma guarantees that the ratio $\frac{\chi}{\phi_1}$ extends continuously to the boundary and hence uniformly bounded away from zero and infinity 
in $\bar \Omega$. Thus,   there are two constants $m$ and $M$ such that $0<   m \le \frac{\chi}{\phi_1} \le  M<\infty$ in $\bar \Omega.$ Moreover, $\lambda^*$ is bounded 
away from infinity uniformly in $A$, see property {\bf iv)} in Section \ref{s:2}. Setting
$M^*=\max_{\bar \Omega} \left(\frac{ \lambda^{\ast}}{\mu_1}\right)\left( \frac{\chi}{\phi_1}\right)f^{\prime}(u^{\ast})$ and choosing $b=\frac{1}{M^*},$ we conclude from 
\eqref{eq:T7} that for $\eps>0$ sufficiently small, we have
\begin{equation}\label{eq:T8}
\left\{  
\begin{aligned}
&P(w)\ge 
\frac{\eps}{M^*} + \mathcal{O}(\eps^2)> \frac{\eps}{2M^*}  &&\text{in } \Omega, \\
&w=0 &&\text{on } \partial\Omega.
\end{aligned}
\right.
\end{equation}
It follows from   \eqref{eq:T5} and \eqref{eq:T8} 
that $w=u^{\ast}+\eps(\phi_1+\chi/M^*)$
 verifies  
\begin{equation}\label{eq:T8_a}
\left\{  
\begin{aligned}
&-\Delta w +A{\bf v}\cdot \nabla w>\lambda^{\ast}f(w) &&\text{in } \Omega, \\
&w=0 &&\text{on } \partial\Omega.
\end{aligned}
\right.
\end{equation}
Thus, it  is a strict super-solution for \eqref{eq:i1} with $\lambda=\lambda^*,$ which is a contradiction by Lemma \ref{l:p5}. Hence, $\mu_1=0.$

%
\end{proof}

The claim  of  Lemma \ref{l:33} is expected and is well known in the context of the classical Gelfand problem. Indeed, non-degeneracy of linearization indicates that the branch
of minimal solutions can be prolonged for larger values of $\lambda.$ This follows directly from the implicit function theorem \cite{Brezis2,Korman}.
 Hence, non-degeneracy of linearization  would contradict the definition of
$\lambda^*.$
It is important to emphasize, however,  that this intuition fails when the  extremal solution is genuinely weak (does not belong to $L^{\infty}$). 
In this case one can have a situation in which the principal eigenvalue of the linearization on an extremal solution is in fact positive 
 \cite{Brezis2}.

To this end, we have established that  the extremal solution $u^*$ is the unique classical solution of \eqref{eq:i1} with $\lambda=\lambda^*(A).$
Moreover,  the linearization of \eqref{eq:i1} on $u^*$ is always degenerate. These results, however, are not sufficient  to establish the main result of this 
paper which requires continuity and differentiability of the extremal Frank-Kamenetskii parameter  $\lambda^*(A)$
which will be established below.



\begin{lemma}\label{l:34} The   extremal value of Frank-Kamenetskii parameter  $\lambda^*(A)$ is a continuous and differentiable function.
Moreover,
\begin{eqnarray}\label{eq:dotlam}
 \dot  \lambda^*(A):= \frac{d}{d A}  \lambda^*(A)=\frac{\int_{\Omega}( {\bf v} \cdot \nabla u^* )\tilde \phi_1}{\int_\Omega f(u^*) \tilde \phi_1},
  \end{eqnarray}
where $u^*$ is the extremal solution of \eqref{eq:i1} and
$\tilde \phi_1>0$ in $\Omega$  is an eigenfunction corresponding to the principal eigenvalue of an operator adjoint to the linearization of \eqref{eq:i1} about $u^*$. That
is a positive solution of
\begin{eqnarray}\label{eq:im3}
\left\{
\begin{array}{lll}
-\Delta \tilde \phi_1-A{\bf v} \cdot \nabla \tilde \phi_1 -\lambda^*(A)  f^{\prime} (u^*) \tilde \phi_1 =0& \mbox{in} & \Omega, \\
\tilde \phi_1=0 & \mbox{on} & \partial \Omega.
\end{array}
\right.
\end{eqnarray}
\end{lemma}
\begin{proof} The proof consists of several steps.
 Let us first prove that $\lambda^*(A)$ is Lipschitz continuous.  Fix $A_0\in \mathbb{R}$ and let $\eps>0$ be  sufficiently small.  Let $u^0$ and $u^{\eps}$ be the extremal solutions
of \eqref{eq:i1} with $A=A_0$ and $A=A_0+\eps$ respectively, and $\lambda_0^*:=\lambda^*(A_0), ~ \lambda_\eps^*=\lambda^*(A_0+\eps).$  That is, 
\begin{eqnarray}\label{eq:l34_a}
\left\{
\begin{array}{lll}
-\Delta u^0+A_0 {\bf v} \cdot \nabla u^0=\lambda_0^* f(u^0) & \mbox{in} & \Omega,\\
-\Delta u^{\eps}+(A_0+\eps)  {\bf v} \cdot \nabla u^{\eps}=\lambda_\eps^* f(u^{\eps}) & \mbox{in} & \Omega,\\
u^0=u^{\eps}=0 & \mbox{on} & \partial \Omega.
\end{array}
\right.
\end{eqnarray}
We recall that $u^0, u^{\eps}$ are classical solutions of \eqref{eq:l34_a} as follow from Lemma \ref{l:33}.

Set $c_0=\sup_{\bar \Omega} |{\bf v} \cdot \nabla u^{\eps} |$ and $c_1=\frac{c_0}{f(0) \lambda_{\eps}^*}.$ We then have from the second equation in \eqref{eq:l34_a} 

\begin{eqnarray}
-\Delta u^{\eps}+A_0  {\bf v} \cdot \nabla u^{\eps}=\lambda_{\eps}^* f(u^{\eps})-\eps {\bf v}\cdot \nabla u^{\eps}\ge \lambda_{\eps}^* f(u^{\eps})-\eps c_0 \ge \lambda_{\eps}^*(1-\eps c_1)  f(u^{\eps}).
\end{eqnarray}
Therefore, $u^{\eps}$ is a classical super-solution for
\begin{eqnarray}\label{eq:add1}
\left\{
\begin{array}{lll}
-\Delta u+A_0  {\bf v} \cdot \nabla u = \lambda_{\eps}^*(1-\eps c_1)  f(u) &\mbox{in} & \Omega,\\
u=0 & \mbox{on} & \partial \Omega.
\end{array}
\right.
\end{eqnarray}
Hence, by Lemma \ref{l:p1}, \eqref{eq:add1} admits a classical solution and thus, by the definition of $\lambda_0^*$ we must have
\begin{eqnarray}\label{eq:l34_b}
\lambda_0^*\ge \lambda_{\eps}^*-c_2\eps,
\end{eqnarray}
 where $c_2=\frac{c_0}{f(0)}.$

On the other hand, as follows from the first equation in \eqref{eq:l34_a}, we have
 \begin{eqnarray}
-\Delta u^0+(A_0+\eps)  {\bf v} \cdot \nabla u^0=\lambda_0^* f(u^0)+\eps {\bf v}\cdot \nabla u^0\ge \lambda_0^* f(u^0)-\eps c_3 \ge \lambda_0^*(1-\eps c_4)  f(u^0),
\end{eqnarray}
where $c_3=\sup_{\bar \Omega} |{\bf v} \cdot \nabla u^0|$ and $c_4=\frac{c_3}{f(0) \lambda_0^*}.$

Consequently, $u^0$ is a classical  super-solution for 
\begin{eqnarray}
\left\{
\begin{array}{lll}
-\Delta u+(A_0 +\eps)  {\bf v} \cdot \nabla u = \lambda_0^*(1-\eps c_4)  f(u) &\mbox{in} & \Omega,\\
u=0 & \mbox{on} & \partial \Omega,
\end{array}
\right.
\end{eqnarray}
which, by Lemma \ref{l:p1}, implies that 
\begin{eqnarray}\label{eq:l34_c}
\lambda_\eps^*\ge \lambda_0^*-c_5\eps,
\end{eqnarray}
where $c_5=\frac{c_3}{f(0)}.$

Combining \eqref{eq:l34_b} and \eqref{eq:l34_c}, we have $|\lambda_\eps^*-\lambda_0^*| \le C \eps$, where the constant $C$ is independent of $\eps$. 
In view that the computations above are symmetric with respect to changing the sign of $\eps,$ we conclude that
\begin{eqnarray}
\left | \frac{\lambda^*(A_0+\eps) -\lambda^* (A_0)}{\eps} \right |\le C.
\end{eqnarray}
Hence, $\lambda^*(A)$ is Lipschitz continuous as claimed.

In what follows, we denote
\begin{eqnarray}\label{eq:sigma}
\sigma(\eps):= \frac{ \lambda_{\eps}^* -\lambda_0^*}{\eps}=\frac{\lambda^*(A_0+\eps) -\lambda^* (A_0)}{\eps},
\end{eqnarray}
with $\sigma(\eps)$ being some uniformly bounded function.

Next,  taking th difference  of  the first and the second equations in \eqref{eq:l34_a}, setting 
\begin{eqnarray}
w^{\eps}:=u^{\eps}-u^0,
\end{eqnarray}
  and using convexity of $f,$ we have
\begin{eqnarray}
-\Delta w^{\eps}+A_0 {\bf v} \cdot \nabla w^{\eps}=-\eps {\bf v}\cdot \nabla u^{\eps}+\lambda_{\eps}^* f(u^{\eps}) -\lambda_0^* f(u^0)=\nonumber\\
-\eps {\bf v}\cdot \nabla u^{\eps}+(\lambda_{\eps}^*-\lambda_0^*) f(u^{\eps}) +\lambda_0^* ( f(u^{\eps})-f(u^0))=\\
-\eps {\bf v}\cdot \nabla u^{\eps}+\eps \sigma(\eps) f(u^{\eps}) +\lambda_0^* (f^{\prime}(u^0)w^{\eps} +\frac12 f^{\prime\prime} (\xi_1) (w^{\eps})^2), \nonumber
\end{eqnarray}
and 
\begin{eqnarray}
-\Delta w^{\eps} +(A_0+\eps) {\bf v} \cdot \nabla w^{\eps}=-\eps {\bf v} \cdot \nabla u^0+\lambda^*_{\eps}f(u^{\eps}) -\lambda_0^* f(u^0)=\nonumber \\
-\eps {\bf v} \cdot \nabla u^0+(\lambda_{\eps}^*-\lambda_0^*) f(u^0)+ \lambda_{\eps}^* (f(u^{\eps})-f(u^0))=\\
-\eps {\bf v} \cdot \nabla u^0+\eps \sigma(\eps) f(u^0) +\lambda_{\eps}^*(f^{\prime}(u^{\eps})w^{\eps}-\frac12 f^{\prime\prime}(\xi_2) (w^{\eps})^2), \nonumber
\end{eqnarray}
where $\xi_1,\xi_2$ are some intermediate points between $u^{\eps}$ and $u^0.$  Rearranging terms in the two equations above, we have 
\begin{eqnarray}\label{eq:l34_e}
-\Delta w^{\eps}+A_0 {\bf v} \cdot \nabla w^{\eps} -\lambda_0^* f^{\prime} (u^0) w^{\eps}= -\eps {\bf v}\cdot \nabla u^{\eps}+\eps \sigma(\eps) f(u^{\eps}) +\frac{\lambda_0^*}{2} f^{\prime\prime} (\xi_1) (w^{\eps})^2,
\end{eqnarray}
\begin{eqnarray}\label{eq:l34_dd}
-\Delta w^{\eps}+(A_0+\eps)  {\bf v} \cdot \nabla w^{\eps} -\lambda_{\eps}^* f^{\prime} (u^{\eps}) w^{\eps}= -\eps {\bf v}\cdot \nabla u^0+\eps \sigma(\eps) f(u^0) -\frac{\lambda_{\eps}^*}{2} f^{\prime\prime} (\xi_2) (w^{\eps})^2.
\end{eqnarray}
The  linear operators on  the left hand sides of \eqref{eq:l34_e} and \eqref{eq:l34_dd} are not invertible. 
Consequently, by  Fredholm alternative \cite[Section 6.2.3, Theorem 4]{Evans},
the right hand sides of these equations must be
orthogonal in $L^2$ to eigenfunctions corresponding to the principal eigenvalues  of  adjoint operators to linearization of \eqref{eq:i1} about  $u^0$ and $u^{\eps},$ respectively.
That is, the right hand sides  of \eqref{eq:l34_e} and \eqref{eq:l34_dd}  must be orthogonal in $L^2$ to  $\tilde \phi_1^{0}, \tilde \phi_1^{\eps}>0,$ which  verify 
\begin{eqnarray}
\left\{
\begin{array}{lll}
-\Delta \tilde \phi_1^0-A_0 {\bf v} \cdot \nabla \tilde \phi_1^0 -\lambda_0^* f^{\prime} (u^0) \tilde \phi_1^0=0 & \mbox{in} & \Omega,\\
-\Delta \tilde \phi_1^{\eps}-(A_0+\eps)  {\bf v} \cdot \nabla \tilde \phi_1^{\eps}-\lambda_{\eps}^* f^{\prime} (u^{\eps}) \tilde \phi_1^{\eps}=0 & \mbox{in} & \Omega,\\
\tilde \phi_1^0=\tilde \phi_1^{\eps}=0 & \mbox{on} & \partial \Omega.
\end{array}
\right.
\end{eqnarray}
Note that regularity of $u^0$ and $u^{\eps}$ and assumptions (H1) , (H2) guarantee that $\tilde \phi_1^0, \tilde \phi_1^{\eps} \in C^{2,\alpha} (\bar \Omega),$ as follows from
\cite[Theorem 6.15]{GT}. We normalize   $\tilde \phi_1^{0}, \tilde \phi_1^{\eps}$ such that $||\tilde \phi_1^0||_2=||\tilde \phi_1^{\eps}||_2=1,$ which makes 
$\tilde \phi_1^0, \tilde \phi_1^{\eps}$ uniquely defined since  $\tilde \phi_1^0, \tilde \phi_1^{\eps}$  are eigenfunction corresponding to  principal eigenvalues which are simple by \cite[Theorem 2.1]{BNV}.

Hence, the solvability conditions read:
\begin{eqnarray}\label{eq:l34_f}
&& \int_{\Omega} (\sigma(\eps) f(u^{\eps}) -{\bf v}\cdot \nabla u^{\eps})\tilde \phi_1^0=-\frac{\lambda_0^*}{2\eps} \int_{\Omega} f^{\prime\prime}(\xi_1) (w^{\eps})^2 \tilde \phi_1^0, \nonumber\\
&& \int_{\Omega} (\sigma(\eps) f(u^0) -{\bf v}\cdot \nabla u^0)\tilde \phi_1^{\eps}=\frac{\lambda_{\eps}^*}{2\eps} \int_{\Omega} f^{\prime\prime}(\xi_2) (w^{\eps})^2 \tilde \phi_1^{\eps}.
\end{eqnarray}
Equalities above imply 
\begin{eqnarray}\label{eq:l34_g}
 \int_{\Omega} (\sigma(\eps) f(u^{\eps}) -{\bf v}\cdot \nabla u^{\eps})\tilde \phi_1^0\le 0, \quad  \int_{\Omega} (\sigma(\eps) f(u^0) -{\bf v}\cdot \nabla u^0)\tilde \phi_1^{\eps}\ge 0.
\end{eqnarray}
Moreover, since left hand sides in equalities \eqref{eq:l34_f} are bounded independently of $\eps,$ we conclude that their right hand sides must be bounded as well.
Thus, $ \frac{\lambda_0^*}{2\eps} \int_{\Omega} f^{\prime\prime}(\xi_1) (w^{\eps})^2 \tilde \phi_1^0\le C, ~ \frac{\lambda_\eps^*}{2\eps} \int_{\Omega} f^{\prime\prime}(\xi_2) (w^{\eps})^2 \tilde \phi_1^{\eps}\le C.$ Consequently,
\begin{eqnarray}\label{eq:34_s}
 \int_{\Omega}  (w^{\eps})^2 d_{\partial \Omega} =\int_{\Omega}  (u^{\eps}-u^0)^2 d_{\partial \Omega} \le C \eps,
\end{eqnarray}
where $C$ is independent of $\eps$ and $d_{\partial \Omega}$ stands for the distance to the boundary.

 Let us  now show that  $\int_{\Omega}  |\nabla (u^{\eps}-u^0)|^2 \to 0$ as $\eps\to 0.$
Taking the difference of the second and the first equations in \eqref{eq:l34_a} and using mean value theorem, we have
\begin{eqnarray}
&&-\Delta w^{\eps} +A_0 {\bf v} \cdot \nabla w^{\eps} =-\eps {\bf v} \cdot \nabla u^{\eps}+\lambda_{\eps}^* f(u^{\eps})-\lambda_0^* f(u^0)=\nonumber\\
&& -\eps {\bf v} \cdot \nabla u^{\eps}+(\lambda_{\eps}^*-\lambda_0^*) f(u^0) +\lambda_{\eps}^*(f(u^{\eps})-f(u^0))= \\
&&  -\eps {\bf v} \cdot \nabla u^{\eps}+\eps \sigma(\eps) f(u^0) +\lambda_{\eps}^* f^{\prime} (\xi_3) w^{\eps}, \nonumber 
\end{eqnarray}
where $\xi_3$  is an  intermediate point between $u^0$ and $u^{\eps}$.

Multiplying the expression above by $w^{\eps},$ integrating by parts and taking into account that $w^{\eps}|_{\partial \Omega}=0$  and the incompressibility condition 
of the flow ${\bf v},$ we have
\begin{eqnarray}
\int_{\Omega} |\nabla w^{\eps}|^2=\eps \int_{\Omega} \left \{ \sigma(\eps) f(u^0) -{\bf v}\cdot \nabla u^{\eps}\right\} w^{\eps} +\lambda_\eps^* \int_{\Omega} f^{\prime} (\xi_3) (w^{\eps})^2.
\end{eqnarray}
 The equality above and Cauchy-Schwartz and Poincare inequalities yield
 \begin{eqnarray}\label{eq:34_ss}
 \int_{\Omega} |\nabla w^{\eps}|^2 \le C \left ( \eps \left( \int_{\Omega} (w^{\eps})^2\right)^{1/2} + \int_{\Omega} (w^{\eps})^2\right)
 \le C \left( \eps \left(  \int_{\Omega} |\nabla w^{\eps}|^2 \right)^{1/2}+ \int_{\Omega} (w^{\eps})^2\right).
 \end{eqnarray} 
Using Hardy's inequality \cite[Lemma 50.3]{QS} and Cauchy-Schwartz  inequality,  we also obtain
\begin{eqnarray}
&&\int_{\Omega} (w^{\eps})^2 = \int_{\Omega} \left(\frac{w^{\eps}}{d_{\partial \Omega}} \right) \left(w^{\eps} d_{\partial \Omega}\right) \le \left( \int_{\Omega} \left(\frac{w^{\eps}}{d_{\partial \Omega}}\right)^2 \int_{\Omega}  (w^{\eps})^2 d^2_{\partial \Omega} \right)^{1/2}\le \nonumber \\
&&C \left(\int_{\Omega} |\nabla w^{\eps}|^2 \right)^{1/2} \left(  \int_{\Omega} (w^{\eps})^2 d^2_{\partial \Omega} \right)^{1/2} \le 
C \left(\int_{\Omega} |\nabla w^{\eps}|^2 \right)^{1/2} \left(  \int_{\Omega} (w^{\eps})^2 d_{\partial \Omega} \right)^{1/2}.
\end{eqnarray} 
This inequality together with \eqref{eq:34_s} give
\begin{eqnarray}\label{eq:34_sss}
\int_{\Omega} (w^{\eps})^2 \le C \sqrt{\eps} \left(\int_{\Omega} |\nabla w^{\eps}|^2 \right)^{1/2}.
\end{eqnarray}
Combining \eqref{eq:34_ss} and \eqref{eq:34_sss}, we obtain
\begin{eqnarray}
 \int_{\Omega} |\nabla w^{\eps}|^2\le C \sqrt{\eps} \left( \int_{\Omega} |\nabla w^{\eps}|^2\right)^{1/2}.
\end{eqnarray}
Hence, 
\begin{eqnarray}
 \int_{\Omega} |\nabla (u^{\eps}-u^0)|^2\le C \eps,
\end{eqnarray}
and thus $u^{\eps}$ converges to $u^0$ in $H_0^1(\Omega)$ as claimed.

Now, let us show that the normalized eigenfunction $\tilde \phi_1^{\eps}$ converge strongly in $L^2(\Omega),$ to the normalized eigenfunction $\tilde \phi_1^0$ as $\eps\to 0.$  
First, we observe that
\begin{eqnarray}\label{eq:34_abc}
\int_{\Omega} \nabla \tilde \phi^{\eps}_1 \cdot \nabla \eta -(A_0+\eps) \int_{\Omega} {\bf v} \cdot \nabla \tilde \phi_1^{\eps} \eta -\lambda_{\eps}^* \int_{\Omega} f^{\prime} (u^{\eps}) \tilde \phi_1^{\eps} \eta=0
\quad \forall \eta \in H_0^1(\Omega).
\end{eqnarray}
Choosing $\eta=\tilde \phi_1^{\eps} $ and using the normalization  of $\tilde \phi_1^{\eps},$ we have that $ \int_{\Omega} |\nabla \tilde \phi_1^{\eps}|^2<C$, for some $C$ independent of $\eps$. 
Thus, $\tilde \phi_1^{\eps}$ is uniformly bounded
in $H_0^1(\Omega).$ Consequently,  $\tilde \phi_1^{\eps}$ converges weakly in $H_0^1(\Omega)$ and strongly in $L^2(\Omega)$ as $\eps \to 0$,  see e.g \cite[Section 5.7]{Evans}. Hence, taking a limit $\eps\to 0$ in \eqref{eq:34_abc}, we arrive to the following  limiting equation
\begin{eqnarray}\label{eq:34_j}
\int_{\Omega} \nabla \bar \phi \cdot \nabla \eta -A_0\int_{\Omega} {\bf v} \cdot \nabla\bar  \phi \eta -\lambda_{0}^* \int_{\Omega} f^{\prime} (u^0) \bar \phi\eta=0
\quad \forall \eta \in H_0^1(\Omega).
\end{eqnarray}
In view that $g=\lambda_{0}^*  f^{\prime} (u^0) \bar \phi  \in L^2(\Omega),$ we have that $\bar \phi \in W^{2,2}(\Omega)$, thanks to \cite[Theorem 8.12]{GT}.
By Sobolev Imbedding Theorem, \cite[Theorem 4.12, Part II]{AF},  this implies that $\bar \phi  \in C^{0,\alpha} (\bar \Omega)$ and consequently  $g  \in C^{0,\alpha} (\bar \Omega).$
Note that a problem
\begin{eqnarray}\label{eq:34_jj}
\int_{\Omega} \nabla z \cdot \nabla \eta -A_0\int_{\Omega} {\bf v} \cdot \nabla z  \eta = \int_{\Omega} g  \eta
\quad \forall \eta \in H_0^1(\Omega)
\end{eqnarray}
admits a classical solution $z=\bar z$ as 
follows from \cite[Theorem 6.15]{GT}.  Taking the  difference of \eqref{eq:34_j}  and \eqref{eq:34_jj} with $z=\bar z,$ we have
 \begin{eqnarray}
\int_{\Omega} \nabla (\bar z- \bar \phi  ) \cdot \nabla \eta -A_0\int_{\Omega} {\bf v} \cdot \nabla ( \bar z- \bar \phi)  \eta = 0
\quad \forall \eta \in H_0^1(\Omega).
\end{eqnarray}
Choosing $\eta=  \bar z- \bar \phi $ in the equation above and  using the  incompressibility condition,  we obtain
$
\int_{\Omega}| \nabla ( \bar z -\bar \phi)|^2 = 0.
$
Hence, $\bar  \phi=\bar z.$  Consequently, the limiting solution $\bar \phi$ is classical and hence,  by uniqueness of the classical normalized positive eigenfunction,   equals to  $\tilde \phi_1^0$. Thus, the limiting normalized solution $\bar \phi$ of \eqref{eq:34_j} is unique classical  and $\bar \phi=\tilde \phi_1^0.$

To this end, we established that 
\begin{eqnarray}\label{eq:to0}
\int_{\Omega} |\nabla (u^{\eps} -u^0)|^2 \to 0, \quad  \int_{\Omega}  (\tilde \phi_1^{\eps} -\tilde \phi_1^0)^2 \to 0 \quad \mbox{as} \quad \eps\to 0.
\end{eqnarray}
Now, let us proceed to the  final step.  Let 
\begin{eqnarray}
 I_{\eps}:=\int_{\Omega} (\sigma(\eps) f(u^{0}) -{\bf v}\cdot \nabla u^{0})\tilde \phi_1^0.
\end{eqnarray}
By the first inequality in \eqref{eq:l34_g}, we have
\begin{eqnarray}\label{eq:l34_ggg}
&& 0\ge  \int_{\Omega} (\sigma(\eps) f(u^{\eps}) -{\bf v}\cdot \nabla u^{\eps})\tilde \phi_1^0= \nonumber \\
&& \int_{\Omega} (\sigma(\eps) f(u^{0}) -{\bf v}\cdot \nabla u^{0})\tilde \phi_1^0
 + \int_{\Omega} (\sigma(\eps)( f(u^{\eps})-f(u^0)) -{\bf v}\cdot \nabla (u^{\eps}-u^0))\tilde \phi_1^0 =\\
 &&I_{\eps} +\int_{\Omega} (\sigma(\eps)( f(u^{\eps})-f(u^0)) -{\bf v}\cdot \nabla (u^{\eps}-u^0))\tilde \phi_1^0. \nonumber
\end{eqnarray}
Using Cauchy-Schwartz and Poincare inequalities, we observe  that 
\begin{eqnarray}
&&\left |\int_{\Omega} (\sigma(\eps)( f(u^{\eps})-f(u^0)) -{\bf v}\cdot \nabla (u^{\eps}-u^0))\tilde \phi_1^0\right |= \left |\int_{\Omega} (\sigma(\eps) f^{\prime} (\xi_4) (u^{\eps}-u^0) -{\bf v}\cdot \nabla (u^{\eps}-u^0))\tilde \phi_1^0\right |  \nonumber  \\
 &&\le  |\sigma(\eps)|\left |\int_{\Omega}  f^{\prime} (\xi_4) (u^{\eps}-u^0)\tilde \phi_1^0\right|+ \left | \int_{\Omega} {\bf v}\cdot \nabla (u^{\eps}-u^0)\tilde \phi_1^0\right| 
 \le \\
 && C \left ( \left(\int_{\Omega} (u^{\eps}-u^0)^2 \right)^{1/2}+ \left( \int_{\Omega}| \nabla (u^{\eps}-u^0)|^2\right)^{1/2} \right) \le C \left( \int_{\Omega}| \nabla (u^{\eps}-u^0)|^2\right)^{1/2}, \nonumber
\end{eqnarray}
where $\xi_4$ is between $u^0$ and $u^{\eps}.$
Hence, by \eqref{eq:to0}, the second term on the right hand side of  \eqref{eq:l34_ggg} goes to zero as $\eps\to 0.$
This observation and \eqref{eq:l34_ggg} imply that 
\begin{eqnarray}\label{eq:y}
J_{\eps}\ge  I_{\eps},
\end{eqnarray}
for some  $J_{\eps} \to 0$ as $\eps\to 0.$
From the second inequality  \eqref{eq:l34_g}, we have
\begin{eqnarray}\label{eq:l34_ggg1}
&&0\le  \int_{\Omega} (\sigma(\eps) f(u^0) -{\bf v}\cdot \nabla u^0)\tilde \phi_1^{\eps}= \int_{\Omega} (\sigma(\eps) f(u^0) -{\bf v}\cdot \nabla u^0)\tilde \phi_1^{0}
+\int_{\Omega} (\sigma(\eps) f(u^0) -{\bf v}\cdot \nabla u^0)(\tilde \phi_1^{\eps}-\tilde \phi_1^0)\nonumber \\
&&= I_{\eps} +\int_{\Omega} (\sigma(\eps) f(u^0) -{\bf v}\cdot \nabla u^0)(\tilde \phi_1^{\eps}-\tilde \phi_1^0).
\end{eqnarray}
Since, by \eqref{eq:to0},
\begin{eqnarray}
|\int_{\Omega} (\sigma(\eps) f(u^0) -{\bf v}\cdot \nabla u^0)(\tilde \phi_1^{\eps}-\tilde \phi_1^0)| \le C\left ( \int_{\Omega} (\tilde \phi_1^{\eps}-\tilde \phi_1^0)^2\right)^{1/2}
\to 0 \quad \mbox{as} \quad \eps\to 0,
\end{eqnarray}
we have from \eqref{eq:l34_ggg1}
\begin{eqnarray}\label{eq:x}
I_{\eps}\ge K_{\eps},
\end{eqnarray}
for some $K_{\eps} \to 0$ as $\eps\to 0$.
Taking a limit $\eps\to 0$ in  \eqref{eq:y} and \eqref{eq:x}, we have  $I_{\eps} \to 0$ as $\eps \to 0.$
Consequently, $\sigma(\eps)$ has a limit as $\eps\to 0$  and converges to  $\frac{\int_{\Omega}( {\bf v} \cdot \nabla u^0 )\tilde \phi_1^0}{\int_\Omega f(u^0) \tilde \phi_1^0}.$
By the definition of $\sigma(\eps),$ see equation \eqref{eq:sigma},  this limit is  $\frac{d}{d A}  \lambda^*(A)\Big \vert_{A=A_0}$.
Hence, as $\eps\to 0,$ we have 
\begin{eqnarray}
 \sigma(\eps) \to \frac{d}{d A}  \lambda^*(A)\Big \vert_{A=A_0}=\frac{\int_{\Omega}( {\bf v} \cdot \nabla u^0 )\tilde \phi_1^0}{\int_\Omega f(u^0) \tilde \phi_1^0}.
\end{eqnarray}
In view of an arbitrary choice of $A_0,$ we have \eqref{eq:dotlam}, which completes the proof.
\end{proof}

 \begin{proof}[Proof of Proposition \ref{p:1}]
 Claims 1)-4) of Proposition \ref{p:1} follow from Lemmas \ref{l:31}- \ref{l:34}, respectively.
 \end{proof}
 
\section{Proof of Theorem \ref{t:1} }\label{s:4}

In this section, we present a proof of the main result. Set $\lambda^{\sharp}:=\lambda^*(0)$ and $\dot \lambda^{\sharp}:=\dot \lambda^*(0).$  Recall that by Lemma \ref{l:34}
\begin{eqnarray}\label{eq:m1}
 \dot \lambda^{\sharp}:= \frac{\int_{\Omega} ( {\bf v} \cdot \nabla u)  \phi_1}{\int_\Omega f(u)  \phi_1},
  \end{eqnarray}
  where
  $u$ is  the extremal solution of \eqref{eq:i1} with $A=0$ and $\phi_1>0$ is an eigenfunction corresponding to the principal
  eigenvalue of linearization of   \eqref{eq:i1} with $A=0$ about $u.$ That is,
  \begin{eqnarray}\label{eq:m2}
\left\{
\begin{array}{lll}
-\Delta u=\lambda^{\sharp}  f(u) & \mbox{in} & \Omega, \\
u=0 & \mbox{on} & \partial \Omega,
\end{array}
\right.
\end{eqnarray}
 \begin{eqnarray}\label{eq:m3}
\left\{
\begin{array}{lll}
-\Delta \phi_1 =\lambda^{\sharp}  f^{\prime }(u) \phi_1 & \mbox{in} & \Omega, \\
\phi_1=0 & \mbox{on} & \partial \Omega.
\end{array}
\right.
\end{eqnarray}
Integrating the numerator of \eqref{eq:m1}  by parts, taking  into account the boundary condition for $\phi_1$, and using the definition of the stream function $\psi$ , see (H2), we have
  \begin{eqnarray}\label{eq:ldot}
  \dot \lambda^{\sharp}=-\frac{\int_{\Omega} \psi  ( \nabla u\cdot \nabla^{\perp} \phi_1)}{\int_\Omega f(u)  \phi_1}.
  \end{eqnarray}
  Consequently, $\dot \lambda^{\sharp}=0$ for all stream functions $\psi$ satisfying (H2) if and only if $\nabla u \cdot \nabla^{\perp} \phi_1\equiv 0$ in $\Omega.$
  Let us  show that  it is only possible when $\Omega$ is a disk.
  
  \begin{lemma} \label{l:41}
  Assume that  $\nabla u \cdot \nabla^{\perp} \phi_1\equiv 0$ in $\Omega,$ then $\Omega$ is a disk.
  \end{lemma}
  
  \begin{proof}
  Assume $\nabla u \cdot \nabla^{\perp} \phi_1\equiv 0$ in $\Omega$. Observe that $\nabla u\cdot \nabla^{\perp} \phi_1=\frac{\partial (u,\phi_1)}{\partial(x,y)}$ is 
  the Jacobian of $(u,\phi_1).$  Since, $\frac{\partial (u,\phi_1)}{\partial(x,y)}=0,$ $u$ and $\phi_1$ are functionally dependent. Therefore, locally $\phi_1=G(u)$ for some
  $C^2$ function $G$. The regularity of $G$ follows from  Proposition \ref{p:1}. 
  
  Consider a small neighborhood of the boundary of $\Omega$, namely, $\Omega_{\delta} =\{(x,y)\in \Omega:0< {\rm dist}( (x,y), \partial \Omega)<\delta\}.$ Using Hopf's lemma, 
   we conclude that for $\delta>0$ sufficiently small, the level sets of $u$ and $\phi_1$ are regular and hence $G(u)$ is well defined in $\Omega_{\delta}.$
  We set,
  \begin{eqnarray}
  \phi_1=G(u) \quad \mbox{in} \quad \Omega_{\delta}.
  \end{eqnarray}
Next, observe that the boundary conditions  $u=\phi_1=0$ on $\partial \Omega$ imply  $G(0)=0$.
  Moreover, on $\partial \Omega$ we have $\nabla \phi_1 \cdot \nu=G^{\prime} (0) (\nabla u\cdot \nu).$  In view of  Hopf's lemma, $\nabla \phi_1  \cdot \nu, \nabla u \cdot \nu <0$.   Therefore, $G^{\prime}(0)>0.$
  We next compute,
  \begin{eqnarray}
  \Delta \phi_1=G^{\prime\prime}(u) |\nabla u|^2+G^{\prime} (u) \Delta u.
  \end{eqnarray}
  Substituting this relation into \eqref{eq:m3}, we have
  \begin{eqnarray}
 - G^{\prime\prime}(u) |\nabla u|^2-G^{\prime} (u) \Delta u=\lambda^{\sharp} f^{\prime}(u) G(u).
  \end{eqnarray}
  Using \eqref{eq:m2}, we can rewrite equation above  as
  \begin{eqnarray}
  G^{\prime\prime}(u) |\nabla u|^2=\lambda^{\sharp}(f(u) G^{\prime}(u)-f^{\prime}(u) G(u)).
  \end{eqnarray}
  Restricting this relation to the boundary $\partial \Omega$ and noting that $f(0), G^{\prime}(0)>0$, $G(0)=0,$ $|\nabla u| \Bigr\vert_{\partial \Omega} =|\frac{\partial}{\partial \nu} u |\Bigr \vert_{\partial \Omega}>0$ we conclude that on $\partial \Omega$ the following holds,
  \begin{eqnarray}
  G^{\prime\prime}(0) \left | \frac{\partial}{\partial \nu} u \right|^2=\lambda^{\sharp} f(0) G^{\prime}(0).
  \end{eqnarray}
  Thus, for function $u$ solving \eqref{eq:m2} with properties assumed above, we must have $G^{\prime\prime}(0)>0$
and hence
  \begin{eqnarray}
  \frac{\partial}{\partial \nu} u=-\left(\frac{ \lambda^{\sharp}f(0) G^{\prime}(0)}{G^{\prime\prime}(0)}\right)^{1/2}=-C \quad \mbox{on} \quad \partial \Omega,
  \end{eqnarray}
  where $C>0$ is some constant.
  As a result,  problem \eqref{eq:m2} complemented by the condition above becomes an overdetermined problem  which has a solution if and only if $\Omega$ is a disk, see \cite[Theorem 8.3.2]{Pucci}.
  \end{proof}

We now can prove the main result of this paper.

\begin{proof}[Proof of Theorem \ref{t:1}]
By Lemma \ref{l:41}, for any $\Omega$ which is not a disk, there exists an open set  $\tilde \Omega \subset \Omega$ such that either 
$\nabla u\cdot \nabla^{\perp} \phi_1>0 $ or  $\nabla u \cdot \nabla^{\perp} \phi_1<0$ in $\tilde \Omega.$ Choosing a stream function $\psi \in C^2$ such that  $\psi\ge 0, \psi\not\equiv 0$
compactly supported in $\tilde \Omega,$  we have $\dot \lambda^{\sharp} \neq 0.$  Consequently, for such a choice of the stream function we have $\lambda^*(A)=\lambda^{\sharp}
+\dot \lambda^{\sharp} A+o(A).$ In view that $\dot \lambda^{\sharp}  \ne 0,$ we have obtain \eqref{eq:i3}.
\end{proof}

  \begin{remark} \label{r:1} A natural question to ask is whether there is  an optimal choice of a stream function $\psi$   that maximizes $\dot \lambda^{\sharp}.$
  The answer  to this question is straightforward   if one imposes a higher regularity assumption on the reaction term.   Indeed,
  in view that the expression for  $\dot \lambda^{\sharp}$ is invariant under the transformation $\psi \to \psi+const$ and that the condition ${\bf v}\cdot \nu=0$
  requires $\psi=const$ on the boundary,  one can look for an optimal stream function satisfying $\psi=0$ on $\partial \Omega$ normalized such that $||\psi||_2=1.$
  We claim that the optimal choice is given by
   \begin{eqnarray}
  \psi^{\dagger}= - \frac{(\nabla u \cdot \nabla^{\perp} \phi_1)}{||(\nabla u \cdot \nabla^{\perp} \phi_1)||_{2}},
  \end{eqnarray}
  provided $f\in C^3$. In this case, $u, \phi_1 \in C^3(\bar \Omega)$ as follows from  \cite[Theorem  6.19]{GT} and hence 
  $(\nabla u \cdot \nabla^{\perp} \phi_1) \in C^2(\bar \Omega).$ Moreover, as follows from direct computations, $\psi^{\dagger}\Big \vert_{\partial \Omega}=0.$
  Hence, all of the conditions in (H2) are satisfied and $\psi^{\dagger}$ is an admissible function. The optimality of $\psi^{\dagger}$ then follows directly from the equality  case in the  Cauchy-Schwartz inequality.
  Thus, the largest value of $\dot \lambda^{\sharp}$ among all stream functions vanishing at the boundary with $||\psi||_2=1$ is
  \begin{eqnarray}
  \dot \lambda^{\dagger}=
\frac{\left(\int_{\Omega} |\nabla u \cdot \nabla^{\perp} \phi_1|^2\right)^{1/2}}{
  \int_{\Omega} f(u) \phi_1}.
   \end{eqnarray}
 \end{remark}
  
\begin{remark} It is important to note that the fast growth condition of the reaction term $f$ stated in (H1) can be relaxed. The condition of the fast growth is only used 
in establishing the regularity of the extremal solution through an estimate in Lemma \ref{l:p3}. The statement of this lemma remains valid if one
assumes that $f$ is $C^2$ positive strictly increasing convex function satisfying $\int_0^{\infty} \frac{ds}{f(s)} <\infty$ and has  the following  property.
There exists $c_0 \in (0,1),$  $c_1>0$  and $t_0\in (0,\infty)$ such that $f(t_2)>c_1 f(t_1)$,  $t_2>t_1>t_0$ implies $(1-c_0) f^{\prime}(t_2) \ge f^{\prime} (t_1)$
see \cite[ Proof of Theorem 3]{GMN}.
It is straightforward to verify that this assumption holds for most typical nonlinearities such as $f(u)=e^u$ and $f(u)=(1+u)^p$ with $p>1.$
\end{remark}




\bigskip
\noindent {\bf Acknowledgment:} This work was supported in a part by US-Israel Binational Science Foundation via  grant  BSF 2024033.
The authors would like to thank Fedor Nazarov for multiple enlightening discussions.

\bigskip

\noindent {\bf Declaration of Competing Interest:} The authors declare that they have no competing interests.

\bigskip

\noindent {\bf Data availability:}  No data was used or produced for the research described in the article.


\begin{bibdiv} 
\begin{biblist}


\bib{Gelfand}{article}{
   author={Gel\cprime fand, I. M.},
   title={Some problems in the theory of quasilinear equations},
   journal={Amer. Math. Soc. Transl. (2)},
   volume={29},
   date={1963},
   pages={295--381},
}

\bib{Dupaigne}{book}{
   author={Dupaigne, Louis},
   title={Stable solutions of elliptic partial differential equations},
   series={Chapman \& Hall/CRC Monographs and Surveys in Pure and Applied
   Mathematics},
   volume={143},
   publisher={Chapman \& Hall/CRC, Boca Raton, FL},
   date={2011},
   doi={10.1201/b10802},
}


\bib{FK39}{article}{
author= {Frank-Kamenetskii,D. A. },
title={Temperature distribution in the reaction vessel and the stationary theory of thermal explosion},
journal={Journal of Physical Chemistry},
volume={13},
number={6},
year={1939},
pages={738—755},
}

\bib{FK}{book} { AUTHOR = {Frank-Kamenetskii,D. A{.} }, TITLE
= {Diffusion and heat transfer in chemical kinetics}, PUBLISHER
= { Plenum Press}, ADDRESS = {New York}, YEAR = {1969},

}


\bib{ZBLM}{book}{
   author={Zel\cprime dovich, Ya. B.},
   author={Barenblatt, G. I.},
   author={Librovich, V. B.},
   author={Makhviladze, G. M.},
   title={The mathematical theory of combustion and explosions},
   publisher={Consultants Bureau [Plenum], New York},
   date={1985},
}


\bib{Fujita}{article}{
   author={Fujita, Hiroshi},
   title={On the nonlinear equations $\Delta u+e\sp{u}=0$ and $\partial
   v/\partial t=\Delta v+e \sp{v}$},
   journal={Bull. Amer. Math. Soc.},
   volume={75},
   date={1969},
   pages={132--135},
   doi={10.1090/S0002-9904-1969-12175-0},
}

\bib{Brezis1}{article}{
   author={Brezis, Ha\"im},
   author={Cazenave, Thierry},
   author={Martel, Yvan},
   author={Ramiandrisoa, Arthur},
   title={Blow up for $u_t-\Delta u=g(u)$ revisited},
   journal={Adv. Differential Equations},
   volume={1},
   date={1996},
   number={1},
   pages={73--90},
}


\bib{KC67}{article}{
   author={Keller, Herbert B.},
   author={Cohen, Donald S.},
   title={Some positone problems suggested by nonlinear heat generation},
   journal={J. Math. Mech.},
   volume={16},
   date={1967},
   pages={1361--1376},
}


\bib{KK74}{article}{
   author={Keener, J. P.},
   author={Keller, H. B.},
   title={Positive solutions of convex nonlinear eigenvalue problems},
   journal={J. Differential Equations},
   volume={16},
   date={1974},
   pages={103--125},
   doi={10.1016/0022-0396(74)90029-1},
}



\bib{CR75}{article}{
   author={Crandall, Michael G.},
   author={Rabinowitz, Paul H.},
   title={Some continuation and variational methods for positive solutions
   of nonlinear elliptic eigenvalue problems},
   journal={Arch. Rational Mech. Anal.},
   volume={58},
   date={1975},
   number={3},
   pages={207--218},
   doi={10.1007/BF00280741},
}




\bib{Brezis2}{article}{
   author={Brezis, Haim},
   author={V\'azquez, Juan Luis},
   title={Blow-up solutions of some nonlinear elliptic problems},
   journal={Rev. Mat. Univ. Complut. Madrid},
   volume={10},
   date={1997},
   number={2},
   pages={443--469},
}

\bib{Nedev}{article}{
   author={Nedev, Gueorgui},
   title={Regularity of the extremal solution of semilinear elliptic
   equations},
   journal={C. R. Acad. Sci. Paris S\'er. I Math.},
   volume={330},
   date={2000},
   number={11},
   pages={997--1002},
   doi={10.1016/S0764-4442(00)00289-5},
}

\bib{Martel}{article}{
   author={Martel, Yvan},
   title={Uniqueness of weak extremal solutions of nonlinear elliptic
   problems},
   journal={Houston J. Math.},
   volume={23},
   date={1997},
   number={1},
   pages={161--168},
}



\bib{Cabre}{article}{
   author={Cabr\'e, Xavier},
   author={Figalli, Alessio},
   author={Ros-Oton, Xavier},
   author={Serra, Joaquim},
   title={Stable solutions to semilinear elliptic equations are smooth up to
   dimension 9},
   journal={Acta Math.},
   volume={224},
   date={2020},
   number={2},
   pages={187--252},
   doi={10.4310/acta.2020.v224.n2.a1},
}




\bib{Law}{book}{
author={ Law, C. K. },
title={Combustion Physics},
publisher={Cambridge University Press},
address={Cambridge},
year={2010}
}


\bib{Sem}{article}{
author={Semenov, N.N.},
title={Thermal theory of combustion and explosion},
Journal={Physics-Uspekhi},
Volume={23},
Issue={3},
year={1940},
pages={251-292},
}

\bib{Will}{book} {
AUTHOR = {Williams, F.},
     TITLE = {Combustion theory},
 PUBLISHER = {Perseus Books},
   ADDRESS = {Reading, MA},
      YEAR = {1985},
    }





\bib{Kagan}{article}{
  author={Kagan, L},
   author={Berestycki, H},
   author={Joulin, G},
   author={Sivashinsky, G},
   title={The effect of stirring on the limits of thermal explosion},
   journal={Combust. Theory and Modelling},
   volume={1},
   date={1997},
   pages={97--11},
}





\bib{Kiselev}{article}{
   author={Berestycki, Henri},
   author={Kiselev, Alexander},
   author={Novikov, Alexei},
   author={Ryzhik, Lenya},
   title={The explosion problem in a flow},
   journal={J. Anal. Math.},
   volume={110},
   date={2010},
   pages={31--65},
   doi={10.1007/s11854-010-0002-7},
}

\bib{Novikov}{article}{
   author={Novikov, Alexei},
   title={On the explosion problem in a ball},
   journal={Commun. Math. Sci.},
   volume={13},
   date={2015},
   number={4},
   pages={1025--1032},
   doi={10.4310/CMS.2015.v13.n4.a9},
}


\bib{Iyer}{article}{
   author={Iyer, Gautam},
   author={Novikov, Alexei},
   author={Ryzhik, Lenya},
   author={Zlato\v s, Andrej},
   title={Exit times of diffusions with incompressible drift},
   journal={SIAM J. Math. Anal.},
   volume={42},
   date={2010},
   number={6},
   pages={2484--2498},
   doi={10.1137/090776895},
}





\bib{eigenvalue}{article}{
   author={Berestycki, Henri},
   author={Hamel, Fran\c cois},
   author={Nadirashvili, Nikolai},
   title={Elliptic eigenvalue problems with large drift and applications to
   nonlinear propagation phenomena},
   journal={Comm. Math. Phys.},
   volume={253},
   date={2005},
   number={2},
   pages={451--480},
   doi={10.1007/s00220-004-1201-9},
}

\bib{Satt}{article}{
   author={Sattinger, D. H.},
   title={Monotone methods in nonlinear elliptic and parabolic boundary
   value problems},
   journal={Indiana Univ. Math. J.},
   volume={21},
   date={1971/72},
   pages={979--1000},
   doi={10.1512/iumj.1972.21.21079},
}




\bib{BNV}{article}{
   author={Berestycki, H.},
   author={Nirenberg, L.},
   author={Varadhan, S. R. S.},
   title={The principal eigenvalue and maximum principal for second-order
   elliptic operators in general domains},
   journal={Comm. Pure Appl. Math.},
   volume={47},
   date={1994},
   number={1},
   pages={47--92},
   doi={10.1002/cpa.3160470105},
}

\bib{GT}{book}{
   author={Gilbarg, David},
   author={Trudinger, Neil S.},
   title={Elliptic partial differential equations of second order},
   series={Classics in Mathematics},
   publisher={Springer-Verlag, Berlin},
   date={2001},
}


\bib{semistab}{article}{
   author={Cowan, Craig},
   author={Ghoussoub, Nassif},
   title={Regularity of the extremal solution in a MEMS model with
   advection},
   journal={Methods Appl. Anal.},
   volume={15},
   date={2008},
   number={3},
   pages={355--360},
   doi={10.4310/MAA.2008.v15.n3.a7},
}

\bib{LYZ}{article}{
   author={Luo, Xue},
   author={Ye, Dong},
   author={Zhou, Feng},
   title={Regularity of the extremal solution for some elliptic problems
   with singular nonlinearity and advection},
   journal={J. Differential Equations},
   volume={251},
   date={2011},
   number={8},
   pages={2082--2099},
   doi={10.1016/j.jde.2011.07.011},
}


\bib{GMN}{article}{
   author={Gordon, Peter V.},
   author={Moroz, Vitaly},
   author={Nazarov, Fedor},
   title={Gelfand-type problem for turbulent jets},
   journal={J. Differential Equations},
   volume={269},
   date={2020},
   number={7},
   pages={5959--5996},
   doi={10.1016/j.jde.2020.04.026},
}


\bib{Evans}{book}{
   author={Evans, Lawrence C.},
   title={Partial differential equations},
   series={Graduate Studies in Mathematics},
   volume={19},
   edition={2},
   publisher={American Mathematical Society, Providence, RI},
   date={2010},
   pages={xxii+749},
   isbn={978-0-8218-4974-3},
   doi={10.1090/gsm/019},
}

\bib{AF}{book}{
   author={Adams, Robert A.},
   author={Fournier, John J. F.},
   title={Sobolev spaces},
   series={Pure and Applied Mathematics (Amsterdam)},
   volume={140},
   edition={2},
   publisher={Elsevier/Academic Press, Amsterdam},
   date={2003},
   pages={xiv+305},
   isbn={0-12-044143-8},
}

\bib{Korman}{book}{
   author={Korman, Philip},
   title={Global solution curves for semilinear elliptic equations},
   publisher={World Scientific Publishing Co. Pte. Ltd., Hackensack, NJ},
   date={2012},
   doi={10.1142/8308},
}


\bib{QS}{book}{
   author={Quittner, Pavol},
   author={Souplet, Philippe},
   title={Superlinear parabolic problems},
   series={Birkh\"auser Advanced Texts: Basler Lehrb\"ucher. [Birkh\"auser
   Advanced Texts: Basel Textbooks]},
   edition={2},
   note={Blow-up, global existence and steady states},
   publisher={Birkh\"auser/Springer, Cham},
   date={2019},
   pages={xvi+725},
   doi={10.1007/978-3-030-18222-9},
}



\bib{Pucci}{book}{
   author={Pucci, Patrizia},
   author={Serrin, James},
   title={The maximum principle},
   series={Progress in Nonlinear Differential Equations and their
   Applications},
   volume={73},
   publisher={Birkh\"auser Verlag, Basel},
   date={2007},
   pages={x+235},
   isbn={978-3-7643-8144-8},
}











\end{biblist}
\end{bibdiv}
\end{document}